\documentclass[12pt]{article}

\usepackage{amsmath,amsthm,amssymb,comment}
\makeatletter

\newtheorem{thm}{Theorem}[section]
\newtheorem{lem}[thm]{Lemma}

\newtheorem{pro}[thm]{Proposition}
\newtheorem{cor}[thm]{Corollary}

\newcommand{\Sect}[1]{\setcounter{equation}{0}\section{#1}\quad}

\@addtoreset{equation}{section}

\makeatother

\newcommand{\dint}{\displaystyle \int}

\newcommand{\Proof}{\noindent {\bf Proof.}\ }
\newcommand{\Qed}{\quad\hbox{\rule[-2pt]{3pt}{6pt}}\par\bigskip}
\newcommand{\R}{\mathbb{R}}

\newcommand{\eps}{\varepsilon}

\newcommand{\tu}{\tilde u}

\newcommand{\N}{{\mathbb{N}}}

\newcommand{\T}{{\cal T}}

\theoremstyle{definition}

\usepackage{color}

\begin{document}

\title{Isolated singularities in the heat equation behaving \\
like fractional Brownian motions}

\author{
Mikihiro Fujii \\
Graduate School of Mathematics,
Kyushu University, \\
Fukuoka 819-0395, Japan
\\  \ \\
 Izumi Okada \\
Faculty of Mathematics,
Kyushu University,\\
Fukuoka 819-0395, Japan
\\  \ \\
and \\
\\
 Eiji Yanagida  \\
Department of Mathematics,
Tokyo Institute of Technology,\\
Meguro-ku, Tokyo 152-8551, Japan
}

\date{}

\maketitle


\newpage

\begin{abstract}
We consider solutions of the linear heat equation in $\R^N$
with isolated singularities.
It is assumed that the position  of a singular point depends on time
and  is H\"older continuous with the  exponent $\alpha \in (0,1)$.
We show that any isolated singularity is removable if it is weaker than a certain order
depending on $\alpha$.  We also show the optimality of the
removability condition by showing the existence of a solution
with a nonremovable singularity.
These results are applied to the case where
the singular point behaves like a
fractional Brownian motion with the Hurst exponent $H \in (0,1/2] $.
It turns out that $H=1/N$ is critical.
\end{abstract}

\

\noindent
{\bf Key words:}  removability,  isolated singularity, heat equation, fractional
Brownian motion.

\

\noindent
{\bf Abbreviated title:} Singularities in the heat equation

\

\noindent
{\bf AMS Subject Classification 2010:}
35K05, 
35B33, 
35A21, 
60G22  

\

%

\ \\  \

\newpage

\Sect{Introduction}\label{sect:intro}
In the field of partial differential equations, it is an important
and interesting problem
to study singularities of solutions.
For instance, let us consider the Laplace equation
$\Delta u=0$ in $\Omega \setminus \{\xi_0 \}$,
where $\Omega$ is a domain in $\R^N$ and $\xi_0 \in \Omega$.
We say that a singularity of $u$ at the point $x=\xi_0$ is removable
if there exists a classical solution $\tu$ of the
Laplace equation in $\Omega$ such that
$\tu \equiv u$ in $\Omega \setminus \{\xi_0 \}$.
It is well known that the singularity of a solution $u$ at $x=\xi_0$
is removable  if $u(x)=o(|x|^{2-N})$ for $N \ge 3$
and $u(x)=o(\log(1/ |x|))$ for $N=2$ as $x \to \xi_0$.
This condition is optimal, because the fundamental solution of
the Laplace equation is given by
\[
\Psi(x) :=\left\{
\begin{aligned}
&C_N |x|^{2-N}&&\text{if } N\geq3,
\\
&   C_2  \log(1/ |x|) && \text{if } N=2,
\end{aligned} \right.
\]
where $C_N>0$ denotes a constant depending on $N$.
Similar results have been obtained for nonlinear elliptic equations as well,
see, e.g., \cite{BV,GS,PLL,V,Vmono} and references cited therein.
On the other hand, for the heat equation
$u_t=\Delta u$ in $( \R^N \setminus \{\xi_0 \} ) \times (0,T)$
with $N\geq 3$ and $T>0$,
Hsu \cite{Hsu} and Hui \cite{Hui} proved that the singular point $\xi_0$ is
removable if and only if $|u(x,t)|=o(\Psi(x-\xi_0)) $ as  $x\to \xi_0$
  locally uniformly on some time interval.
This result is optimal  since  $\Psi$ is a stationary solution of  the heat equation.

In this paper, we study the heat equation in the case where the position of
a singular point depends on time, which  is formulated as
\begin{equation} \label{eq:main}
	u_t = \Delta u,  \qquad  (x,t) \in D,
\end{equation}
where
\[
D:= \left\{ (x,t) \in \R^{N+1} \,:\,
		 x\in \R^N \setminus\{\xi(t)\}, \; t\in(0,T)\right\}.
\]
Throughout of this paper, we assume that $N\geq3$
and  the singular point $\xi:\R \rightarrow \R^N$  is continuous.
Our interest for (\ref{eq:main}) is in the  removability  of a singularity
and the existence of a positive singular solution,
especially when the singular point behaves like a
fractional Brownian motion with the Hurst exponent $H \in (0,1/2]$.
See \cite{MD, YK, YK2} for related results concerning
the fractional Brownian motion with the Hurst exponent $H \in (0,1/N]$.

For a solution $u$ of (\ref{eq:main}),  the singularity at $x=\xi(t)$ is said to be removable
if there exists a function $\tu$ which satisfies the heat equation
in $\R^N \times(0,T)$ in the classical sense and  $\tu \equiv u$ on $D$.
Takahashi-Yanagida~\cite{TakahashiY} studied the case where
 $\xi(t)$ is  locally $1/2$-H\"older continuous
in $t\in\R$ and showed that if  a solution $u$ of \eqref{eq:main} satisfies
$
u(x,t)=o(\Psi(x-\xi(t))) $ as  $ x\to  \xi(t)$
 locally uniformly in $t\in(0,T)$,
then the singularity of $u$ at $x=\xi(t)$ is removable.
They  also showed the existence of
solutions with nonremovable time-dependent singularities
that are asymptotically radially symmetric
(see \cite{Hirata,SY,TY2} for related results about time-dependent singularities in
nonlinear parabolic equations).
Kan-Takahashi~\cite{KT} studied the case where the limit
\[
 \lim_{t \uparrow T}  \frac{ \xi(T) - \xi(t) }{(T - t)^{\alpha}}
\]
exists for some $\alpha \in (0,1/2)$, and showed the existence of a solution
whose limiting profile loses the asymptotic radial symmetry.

Our first objective is to extend the removability condition to the case where
$\xi(t)$ is $\alpha$-H\"older continuous with $\alpha \in (0,1)$.
Taking the result of \cite{KT} into account, we state our result
by using the integral of $u(x,t)$ on a ball
\[
B^r(t):= \{x \in \R^N : |x - \xi(t)| \leq r \}.
\]

\begin{thm}[Removability]\label{th:remove}
Let $N \geq 3$ and $\alpha \in (1/N,1)$.
Suppose that $\xi(t)$ is  locally $\alpha$-H\"older continuous
in $t\in(0,T)$.    If a positive solution $u$ of \eqref{eq:main} satisfies
\[
\int_{B^r(t)} u(x,t)dx =
\left \{ \begin{aligned}
&  o(r^2) &&  \mbox{ for }  1/2 \leq \alpha <1 ,
\\
&o(r^{1/\alpha}) &&   \mbox{ for }  1/N<\alpha < 1/2
 \end{aligned} \right.
\]
as $r\to 0$ locally uniformly in $t\in(0,T)$,
then the singularity of $u$ at $x=\xi(t)$ is removable.
\end{thm}

If a solution $u$ is positive and bounded in a neighborhood of $\xi(t)$,
then the solution satisfies
\[
c_1 r^N \leq \int_{B^r(t)} u(x,t) dx \leq c_2 r^N, \qquad r \in (0,1)
\]
with some constant $c_1,c_2>0$.
This implies that  for $0<\alpha \leq 1/N$,
we cannot judge the removability from the integral on $B^r(t)$.
In fact, we will see that for every $ 0<\alpha \leq 1/N$,
there exists a solution with a singularity at
$\xi(t)$ that is bounded  at $\xi(t)$.  Therefore, to study the
removability, we need to examine more precise profile of the solution
around the singular point.

Our idea of the proof of Theorem~\ref{th:remove}
 is to extend the method in \cite{TakahashiY} that is
 based on a construction of a suitable cut-off function.
In order to construct the desired function, the H\"older exponent
$\alpha=1/2$ is critical  in some sense.
Using the cut-off function, we can show under the condition
 in Theorem~\ref{th:remove} that the solution satisfies (\ref{eq:main})
in a weak sense.  Then the parabolic regularity implies that the solution
satisfies the heat equation on $\R^N$ in the classical sense.

Next, we show that the condition in Theorem~\ref{th:remove} is optimal.
To show the optimality, we consider the  initial value problem
\begin{equation} \label{eq:atdelta}
\left \{
\begin{aligned}
	&u_t - \Delta u = \delta(x - \xi(t)),  &\quad &
		 x \in \R^N , \ t  > 0,
\\
	&u(x,0)=u_0(x), &&  x \in \R^N ,
\end{aligned} \right.
\end{equation}
where $\delta(\cdot)$ denotes a Dirac measure concentrated
at the origin $0 \in \R^N$, and the initial value $u_0(x)$
is assumed to be continuous, positive and bounded on $\R^N$.
By solving this problem, we show the existence of a solution of (\ref{eq:atdelta})
with a nonremovable singularity at $\xi(t)$.
In the case of $\alpha>1/2$,
 it was shown in \cite{TakahashiY} that if $\xi(t)$ is $\alpha$-H\"older continuous,
(\ref{eq:atdelta}) has a solution satisfying
\begin{equation} \label{eq:heatasym}
 u(x,t) =  C_N |x - \xi(t)|^{-N+2} + o( |x - \xi(t)|^{-N+2}) \qquad (x \to \xi(t)).
\end{equation}

In this paper, we are particularly interested in the case
where $\xi(t)$ is H\"older continuous with the exponent $\alpha\le 1/2$.
For example, any sample  path  of the fractional Brownian
motion with the Hurst exponent $H \in (0,1/2]$ is ($H - \eps$)-H\"older continuous
  almost everywhere in $t$,
where $\eps>0$ is an arbitrarily small constant.
We note that $H=1/2$ corresponds to the ordinary Brownian motion.
Therefore, it becomes an interesting question to ask
what happens  to the solution of (\ref{eq:atdelta}) depending on $H$.

In order to investigate the behavior of solutions to (\ref{eq:atdelta}) with such  $\xi(t)$,
we first consider the case where the singularity $\xi(t)$ satisfies more general conditions.
Let $s_0 \in (0,T]$ be arbitrarily fixed. For small $r>0$, we define a function
\[
 \sigma(r) := \inf  \{ s\in [0,s_0]:   |\xi(T -s) - \xi(T)| > r \} .
\]
(We set $\sigma(r) = s_0$ if  $|\xi(T -s) - \xi(T)| \le r$ for all $ s\in [0,s_0]$.)
We also define a set
\begin{equation} \label{eq:bigtau}
\T (r):= \{  s \in [0, s_0]:    |\xi(T - s) - \xi(T)|  \leq  r \}
\end{equation}
and  a function
\begin{equation} \label{eq:smalltau}
\tau(r) := \mu (\T (r)),
\end{equation}
where $\mu (\cdot)$ stands for the  Lebesgue measure on $\R$, that is
\begin{equation*}
\tau(r) = \int_{0}^{s_0} 1_{\{ |\xi(T - s) - \xi(T)|  \leq  r \}} ds.
\end{equation*}
Clearly, $\sigma(r) \leq \tau(r)$ for all $r>0$.
In the field of probability theory,
the functions $\sigma(r)$ and $\tau(r)$ are called the first exit time
and the occupation time, respectively.


The following result gives a
lower bound of solutions of (\ref{eq:atdelta}).

\begin{thm}
\label{th:lower}
Let $N \geq 3$.  Assume that $\sigma(r)=O(r^2)$ as $r \to 0$.
Then for any $\theta \in (0,1)$, there exist constants $C>0$
and $R>0$
such that the solution of \eqref{eq:atdelta} satisfies
\[
\int_{B^r(T)}u(x,T) dx  \geq
C \big \{ \sigma(\theta r) +  r^{N}  \big \}
\]
for all $r\in (0,R)$.
\end{thm}

We note that if $\xi(t)$ moves not less slowly
than $(T-t)^{1/2}$ as $t \uparrow T$, then the assumption on $\sigma$
in Theorem~\ref{th:lower}  is satisfied.
More precisely, $\sigma(r)=O(r^2)$ as $r \to 0$ if and only if
\[
\displaystyle \liminf_{s \downarrow 0}
  s^{-1/2}\sup_{T-s<t<T}|\xi(T) - \xi(t)| \in (0,\infty].
\]

Next, we give a pointwise upper bound of solutions of (\ref{eq:atdelta}).

\begin{thm}
 \label{th:upper}
Let $N \geq 3$.
Then there exist constants $C>0$ and $R>0$  such that
the solution $u$ of \eqref{eq:atdelta} satisfies
\[
\int_{B^r(T)} u(x,T) dx  \leq
 C  \left[  \min \left \{  r^N \int_{r}^\infty \tau(l)l^{-N-1}dl ,r^2\right \} + r^{N} \right]
\]
for all  $ r \in [0,R]$.
\end{thm}

We note that if $r^{-\kappa} \tau(r)$ is nonincreasing in $r$ with some constant $\kappa<N$, then
the condition on $\tau$ in Theorem~\ref{th:upper}  is simplified as
\[
\begin{aligned}
r^N \int_{r}^\infty \tau(l)l^{-N-1}dl
&=  r^N \int_{r}^\infty l^{-\kappa}  \tau(l)  l^{-N-1+\kappa}dl\\
\label{b1}
&\leq  r^N  r^{-\kappa}  \tau(r) \int_{r}^\infty l^{-N-1+\kappa}dl
\leq C \tau(r).
\end{aligned}
\]

We apply Theorems~\ref{th:lower} and \ref{th:upper} to some special cases.
We first consider the simple case where $\xi(t)$ is $\alpha$-H\"older continuous at $t=T$
with  $0<\alpha  \leq 1/2$.

\begin{cor} \label{co:bound}
Let $N \geq 3$.   Suppose that  $\xi$ satisfies
\[
  c_1 (T-t)^\alpha \leq |\xi(T) - \xi(t)| \leq  c_2 (T-t)^\alpha, \qquad t \in [t_0,T],
\]
with some $\alpha \in (0,1/2]$, $c_1,c_2>0$ and $t_0 \in [0,T)$.
Then there exist constants $C_1,C_2>0$ and $R>0$ such that
the solution $u$ of \eqref{eq:atdelta} satisfies
the following inequalities for $r \in (0, R]$:
\begin{description}
\item[\rm (i)]
If $1/N < \alpha \leq 1/2$, then
\[
C_1  r^{1/\alpha}  \leq \dint_{B^r(T)} u(x,T) dx  \leq  C_2 r^{1/\alpha}.
\]
\item[\rm (ii)] If $0 < \alpha  \leq 1/N$, then
\[
C_1  r^{N}   \leq  \int_{B^r(T)} u(x,T) dx \leq C_2  r^{N}.
\]
\end{description}
\end{cor}

When $\xi(\cdot)$ is a sample path of the fractional Brownian motion,
then we can apply Theorems~\ref{th:lower} and \ref{th:upper} to
obtain the following result.

\begin{thm}
[Fractional Brownian motion]\label{th:Brown}
Let $N \geq 3$.
Suppose that $\xi(\cdot)$  is a sample path of the fractional Brownian motion
with the Hurst exponent $H$.
Then for every $t \in (0,T]$ and $R>0$,
there exist $C_1(\omega),C_2(\omega)>0$
such that the solution $u$ of \eqref{eq:atdelta} satisfies
the following inequalities  for  $r \in (0,R]$ with probability one:
\begin{description}
\item[\rm (i)] If $1/N< H \le 1/2$, then
\[
C_1 r^{1/H}  \{ \log\log (1/r)\}^{-1/(2H)-\delta}
\leq \dint_{B^r(T)} u(x,t) dx  \leq C_2r^{1/H} \{\log\log (1/r)\}^{1+\delta} .
\]
\item[\rm (ii)] If $H =1/N$, then
\[
C_1  r^N  \leq \dint_{B^r(T)} u(x,t) dx
 \leq C_2  r^{N} ( \log 1/r)^2  \{ \log\log (1/r)\}^{1+\delta}.
\]
\item[\rm (iii)]  If $0<H <1/N$, then
\[
C_1  r^{N}
 \leq \dint_{B^r(T)} u(x,t) dx  \leq C_2 r^{N}\log(1/r) \{ \log \log (1/r)\}^{1+\delta}.
\]
\end{description}
\end{thm}

This paper is organized as follows.
In Section~\ref{sect:removability},
we discuss the removability of singularities.
In Sections~\ref{sect:lower}, we consider (\ref{eq:atdelta})
and derive a lower bound and an upper bound of the solutions
by using the functions $\sigma(r)$ and $\tau(r)$.
In Section~\ref{sect:Brown},
we study the case where $\xi(t)$ is a sample path of
the fractional Brownian motion.

\Sect{Removability of singularities}\label{sect:removability}
In this section,  we give a proof of
Theorem~\ref{th:remove} by usng a suitable cut-off function.
The  following lemma was proved in \cite{TakahashiY}.

\begin{lem}\label{lem:cut} \
Let $N\geq1$, $t_1, t_2\in \R$ $(t_1<t_2)$ and $\alpha\in(0,1]$.
Suppose that $\xi(t)$ is $\alpha$-H\"older continuous
in $t\in [t_1, t_2]$ for some $\alpha\in(0,1]$.
Then there exist $\delta_0=\delta_0(\alpha,N,t_1,t_2)\in(0,1)$
and $C=C(\alpha,N,t_1,t_2)>0$ independent of $x,t$
with the following property:
For every $\delta \in(0,\delta_0)$,
there exists a  cut-off function
$\eta \in  C^\infty(\R^N \times \R)$
such that
\[
\begin{aligned}
& 0\leq \eta(x,t) \leq 1,
\\
&	\eta(x,t)=\left\{
	\begin{split}
	&1  && \text{if } |x-\xi(t)|  \geq \delta,
	\\ &0 && \text{if } |x-\xi(t)| \leq \delta/2,
	\end{split}	\right.
\\
&	|\nabla \eta | \leq C\delta^{-1},\quad
	|\Delta \eta | \leq C\delta^{-2},\quad
	|(\eta)_t | \leq C\delta^{-1/\alpha}
\end{aligned}
\]
for $(x,t)\in \R^N\times[t_1,t_2]$.
\end{lem}

Now we prove Theorem~\ref{th:remove}.
\begin{proof}[Proof of Theorem~\ref{th:remove}]
For  $0<t_1<t_2<T$,
let  $\varphi\in C_0^\infty(\R^N\times(0,T))$ be a test function
with  a compact support in $\R^N \times(t_1,t_2)$.
Then the Weyl lemma for the heat equation
(see, e.g., \cite[Section 6]{G})  implies that
if  $u$ satisfies
\begin{equation}\label{eq:d}
	\int_{t_1}^{t_2}\int_{\R^N}u(\varphi_t +\Delta \varphi) \,dxdt =0
\end{equation}
for any $\varphi$, then it turns out that $u$ actually
belongs to $C^{2,1}(\R^N\times(0,T))$.
Hence, it suffices to prove (\ref{eq:d}) for the removability of
 the singularity of $u$ at $x=\xi(t)$.

Suppose that $u$ satisfies for some $k>0$
\[
\int_{B^r(t)} u(x,t) dx = o(r^k) \qquad {\rm as\ }r\to 0
\]
locally uniformly in $t\in(0,T)$.

Let $\eps\in(0,1)$ be arbitrarily given.  By assumption on $u$, we can take
$\delta_1 = \delta_1(t_1,t_2,\eps)  \in(0,1)$  such that
\[
\int_{B^r(t)} u(x,t)dx  \leq   \eps r^{k},
\qquad 0<r<\delta_1, \ t \in [t_1,t_2] \subset (0,T).
\]
Let $\delta \in ( 0 , \min \{ \delta_0 , \delta_1 \} )$,
where $\delta_0$ is defined in Lemma \ref{lem:cut}.
For this $\delta$, we take a cut off function $\eta\in C_0^{\infty}(\mathbb{R}^N)$
 constructed in Lemma \ref{lem:cut}.
Multiplying \eqref{eq:main} by $\eta\varphi$
and integrating it by parts, we obtain
\begin{equation}\label{eq:de}
		 \int_{t_1}^{t_2} \int_{\R^N}
		 u \{ (\varphi \eta)_t
		 + \Delta(\varphi\eta) \} dx dt
		 =0.
\end{equation}
By simple calculations, we have
\[
\begin{aligned}
	\Big|\int_{t_1}^{t_2}\int_{\R^N}
	u\{ (\varphi \eta)_t-\varphi_t\} dxdt \Big|
 & \leq     \int_{t_1}^{t_2}  \int_{\R^N}
	u \big | (\varphi \eta)_t-\varphi_t \big | dxdt
\\
 & \leq  C_1 (1+\delta^{-1/\alpha})   \int_{t_1}^{t_2}
\dint_{B^{\delta}(t)} u(x,t) dx dt
\\
 & \leq  C_2\eps(1+\delta^{-1/\alpha})\delta^{k}
\end{aligned}
\]
and
\[
\begin{aligned}
	\Big|\int_{t_1}^{t_2}\int_{\R^N}
	u\{ \Delta(\varphi \eta)-\Delta\varphi\} dxdt \Big|
 & \leq     \int_{t_1}^{t_2}  \int_{\R^N}
	u \big | \Delta(\varphi \eta)-\Delta\varphi  \big | dxdt
\\	& \leq
	C_3(1+\delta^{-1}+\delta^{-2}) \int_{t_1}^{t_2}
	\dint_{B^{\delta}(t)} u(x,t) dx    dt	\\
& \leq
	C_4 \eps  (1+\delta^{-1}+\delta^{-2})\delta^{k},
\end{aligned}
\]
where $C_1, C_2, C_3, C_4>0$ are constants independent of $\eps$.
Hence if $k = \max \{2, 1/\alpha \}$,  we obtain
\[
\left|\int_{t_1}^{t_2}\int_{\R^N}u(\varphi_t +\Delta \varphi) \,dxdt\right|  \leq C_5 \eps,
\]
where $C_5>0$ is a constant  independent of $\eps>0$ and $\delta \in (0,\min\{\delta_0,\delta_1\})$.
Since $\eps>0$ is arbitrary, this implies (\ref{eq:d}).
This completes the proof.
\end{proof}

\Sect{Nonremovable singularities}\label{sect:lower}
In this section, we consider the case where the singularity at
$\xi(t)$  is not removable.
Let us consider the initial value problem (\ref{eq:atdelta})  with $u(x,0)=u_0(x)$,
where $u_0(x)$ is continuous positive and bounded on $\R^N$.
By using the heat kernel
\[
G(x,y,t):=\dfrac{1}{(4\pi t)^{N/2}} \exp\Big (- \dfrac{|x-y|^2}{4t} \Big),
\]
the solution of (\ref{eq:atdelta})  is expressed as
\[
	u(x,t)=\int_{\mathbb{R}^N} G(x,y,t)u_0(y)dy + \int_0^t  G(x,\xi(s),t-s) ds
\]
(see \cite{KT,TakahashiY}).
Since the first term of the right hand side is smooth, positive and bounded
in $t>0$, we see that for $R>0$, there exist constants $C_1>0$ and $C_2>0$ depending on $R$ and $u_0$ such that
\[
  C_1r^N \leq \int_{B^r(T)} \int_{\mathbb{R}^N} G(x,y,T)u_0(y)dy dx \leq C_2r^N  \quad
\mbox{ for }  r\in (0,R).
\]
Hence, it suffices to examine the second term
\[
\begin{aligned} F(x,T) : & =\int_0^T  G(x,\xi(t),T-t) dt
\\
 &  =\int_0^T \dfrac{1}{(4\pi (T-t))^{N/2}}
 \exp\Big (- \frac{|x-\xi(t)|^2}{4(T-t)}\Big ) dt.
\end{aligned}
\]
Here,  by exchanging the order of integration and changing the
variables $y=x-\xi(T)$ and $\eta(s)=\xi(T-s) -\xi(T)$, we have
\[
\begin{aligned}
\int_{B^r(T)} F(x,T) dx
& = \int_{|x-\xi(T)|\leq r} \int_0^T
 \dfrac{1}{(4\pi (T-t))^{N/2}} \exp\Big (- \frac{|x-\xi(t)|^2}{4(T-t)}\Big ) dt dx
  \\
 & =  \int_0^T \dfrac{1}{(4\pi s)^{N/2}}
 \int_{B_0^r} \exp\Big (-\frac{| y - \eta(s)  |^2}{4s}\Big)  dy ds,
\\
\end{aligned}
\]
where $B_0^r$ is defined as
\begin{equation*}
  B_0^r:= \{ y : |y|\leq r\}.
\end{equation*}

In the following, $C$ stands for generic constants whose value
 may change from line to line  but does not depend on other variables.

\begin{proof}[Proof of Theorem~\ref{th:lower}]
Let $0<\theta<1$ be arbitrarily given.  By assumption,
$\sigma$ satisfies $\sigma(\theta r) \leq Ar^2$ for all $r\in (0,R)$ with some constant $A>0$ and $R>0$.
For  $s \leq \sigma(\theta r)  \leq A r ^2 $  and a constant $\delta>0$, we have
\[
    \{ y :  |y - \eta(s) | \leq \delta s^{1/2} \}
     \subset    \{y : |y|\leq  \theta r+\delta  s^{1/2} \}
    \subset    \{y : |y|\leq \big (\theta+ \delta A^{1/2} \big) r \}.
\]
Hence, if we take   $\delta>0$ so small that $\theta+\delta A^{1/2} \leq 1$,
then
\[
B_0^r \supset
\{ y:  |y - \eta(s) | \leq \delta s^{1/2}\}.
\]

Now, for  $s \leq \sigma(\theta r) $, we have
\[
\begin{aligned}
 \int_{B_0^r} \exp\Big (-\frac{| y - \eta(s)  |^2}{4s}\Big)  dy
& \geq   \int_{|y - \eta(s) | \leq \delta  s^{1/2}}
  \exp\Big (-\frac{| y - \eta(s)  |^2}{4s}\Big)  dy
\\
& \geq  C  \int_{|y - \eta(s) | \leq \delta  s^{1/2}}   dy
 \geq Cs^{N/2}.
 \end{aligned}
\]
Hence we obtain
\[
\begin{aligned}
\int_{B^r(T)} F(x,T) dx  & \geq
   \int_0^{\sigma(\theta r)}\dfrac{1}{(4\pi s)^{N/2}}
 \int_{B_0^r} \exp\Big (-\frac{| y - \eta(s)  |^2}{4s}\Big)  dy ds
 \\
 & \geq
   \int_0^{\sigma(\theta r)}\dfrac{1}{(4\pi s)^{N/2}}
 \times Cs^{N/2} ds
\\
& \geq  C \sigma(\theta r).
\end{aligned}
\]
This completes the proof.
\end{proof}

\

Before giving a proof of Theorem~\ref{th:upper}, we give
an upper bound of positive solutions.

\begin{lem}\label{le:universal}
For any solution of \eqref{eq:atdelta},
there exists a constant $C>0$ such that
\[
\int_{B^r(T)} u(x,T)dx  \leq C r^2
\]
for $r>0$.
\end{lem}

\begin{proof}
Set $B_1^r(s):= \{y \in \R^N : |y - \eta(s)| \leq r \}$.
We then have
\[
\begin{aligned}
  \int_{B_0^r }&
  \exp\Big (-\frac{|y-\eta(s)|^2}{4s} \Big) dy
\\ &  =   \int_{B_0^r \cap B_1^r(s)} \exp\Big (-\frac{|y-\eta(s)|^2}{4s} \Big) dy
+ \int_{B_0^r \setminus B_1^r(s)} \exp\Big (-\frac{|y-\eta(s)|^2}{4s} \Big) dy
\\ &  \leq    \int_{B_0^r \cap B_1^r(s)} \exp\Big (-\frac{|y-\eta(s)|^2}{4s} \Big) dy
+ \int_{B_1^r(s) \setminus  B_0^r} \exp\Big (-\frac{|y-\eta(s)|^2}{4s} \Big) dy
\\ &  =   \int_{B_1^r(s)} \exp\Big (-\frac{|y-\eta(s)|^2}{4s} \Big) dy
\\ &  =   \int_{B_0^r} \exp\Big (-\frac{|y|^2}{4s} \Big) dy  .
\end{aligned}
\]
Hence, there holds
\[
\begin{aligned}
\int_{B^r(T)} F(x,T)dx &
= \int_0^T    \dfrac{1}{(4\pi s)^{N/2}}\int_{B_0^r}
  \exp\Big (-\frac{|y-\eta(s)|^2}{4s} \Big) dy ds
\\
&    \leq  \int_0^T   \int_{B_0^r } \dfrac{1}{(4\pi s)^{N/2}}
  \exp\Big (-\frac{|y|^2}{4s} \Big) dy ds
\\
& = C \int_0^T   \int_0^r \dfrac{1}{(4\pi s)^{N/2}}
  \exp\Big (-\frac{\rho^2}{4s} \Big) \rho^{N-1} d\rho  ds.
\end{aligned}
\]
Evaluating the last integral, we obtain
\[
\int_{B^r(T)} F(x,T)dx \leq C  r^2.
\]
This completes the proof.
\end{proof}

\begin{proof}[Proof of Theorem~\ref{th:upper}]
We choose $R>0$ so small that $[0,s_0]\setminus \T(2R)$ is not empty and let $r\in (0,R)$.
If $|\eta(s)|\geq 2r$ and $|y|\leq r$, then
\[
  |y-\eta(s)|\geq |\eta(s)|-|y|\geq \frac{1}{2}|\eta(s)|.
\]
Hence we divide the integration into three parts to obtain
\[
  \begin{aligned}
    \int_{B_0^r} F(x,T) dx
    &\leq \int_{s_0}^T \int_{B_0^r} \frac{1}{(4\pi s)^{N/2}} dy ds\\
    &\qquad+\int_{s\in[0,s_0] \setminus \T(2r)}
    \int_{B_0^r} \frac{1}{(4\pi s)^{N/2}} {\rm exp}\left( -\frac{|\eta(s)|^2}{16s} \right) dyds\\
    &\qquad+\int_{s\in \T(2r)}\int_{B_0^r} \frac{1}{(4\pi s)^{N/2}} {\rm exp}
    \left( -\frac{|y-\eta(s)|^2}{4s} \right) dyds\\
    &=:I_1(r)+I_2(t)+I_3(r).
  \end{aligned}
\]
It is easy to see that $I_1(r)$ is bounded by $Cr^N$.
Using the integration by parts for the Lebesgue-Stieltjes measure $d\tau(l)$, we have
\[
  \begin{aligned}
    I_2(r)
    &\leq \int_{s\in[0,s_0] \setminus \T(2r)} \int_{B_0^r} |\eta(s)|^{-N} dyds\\
    &\leq Cr^N\int_{ |\eta(s)|> 2r} |\eta(s)|^{-N} ds\\
    &\leq Cr^N\int_{2r}^{\infty} l^{-N} d\tau(l)\\
    &=-Cr^N (2r)^{-N} \tau(2r)+ Cr^N \int_{2r}^\infty N l^{-N-1} \tau(l) dl   \\
    &\leq Cr^N \int_r^\infty  l^{-N-1} \tau(l) dl.
  \end{aligned}
\]
Finally, since
\[
\int_{B_0^r}
  \exp\Big (-\frac{|y-\eta(s)|^2}{4s} \Big) dy
\leq   \int_{\R^N}
  \exp\Big (-\frac{|y|^2}{4s}\Big) dy
 \leq C  s^{N/2},
\]
the integral $I_3(r)$ satisfies
\[
I_3(r)
\leq  C  \int_{s \in \T(2r)}   ds = C \tau(2 r)  .
\]
Combining these estimates
and using
\[
r^N \int_{2r}^\infty  l^{-N-1} \tau(l) dl \ge r^N \tau(2r) \int_{2r}^\infty  l^{-N-1}  dl  \geq
C\tau(2r),
\]
we obtain
\[
  \begin{aligned}
    \int_{B^r(T)} F(x,T) dx
    &\leq  Cr^N + Cr^N \int_r^\infty  l^{-N-1} \tau(l) dl +C \tau(2 r)\\
    &\leq  Cr^N + Cr^N \int_r^\infty  l^{-N-1} \tau(l) dl .
  \end{aligned}
\]
By Lemma~\ref{le:universal}, the proof is completed.
\end{proof}

\

\begin{proof}[Proof of Corollary~\ref{co:bound}]

First, if  $\xi$ satisfies
\[
   |\xi(T) - \xi(t)| \leq  c_2 (T-t)^\alpha, \qquad t \in [t_0,T],
\]
for some $c_2>0$ and $t_0 \in [0,T)$,
 we have  $\sigma(r) \geq (r/c_2)^{1/\alpha}$.
Then by Theorem~\ref{th:lower},   there exists $C_1>0$
and $R_1>0$
such that
\[
 \dint_{B^r(T)} u(x,T) dx  \geq  C_1 \max\{  r^{1/\alpha}, r^N \}, \qquad r \in (0,R_1).
\]

Next, if
\[
   |\xi(T) - \xi(t)| \geq c_1 (T-t)^\alpha , \qquad t \in [t_0,T],
\]
then $\tau(r) \leq (r/c_1)^{1/\alpha}$.
On the other hand, there exists $c_0>0$ such that $\tau(l ) =  \tau(c_0)$ for $l\ge  c_0$.
Hence, we have
\[
\begin{aligned}
 &r^N \int_r^\infty  l^{-N-1} \tau(l) dl
  \le r^N \int_r^\infty  l^{-N-1+1/\alpha}dl \le Cr^{1/\alpha}
  \quad \text{ if  } \alpha> 1/N,\\
&r^N \int_r^\infty  l^{-N-1} \tau(l) dl
\leq Cr^N \int_0^{c_0}  l^{-N-1+1/\alpha} dl
  +Cr^N \int_{c_0}^{\infty}  l^{-N-1} dl
 \le Cr^{N} \quad \text{ if  } \alpha< 1/N.
\end{aligned}
\]
Then,  by Theorem~\ref{th:upper}, there exist $C_2>0$ and $R_2>0$ such that
\[
  \begin{aligned}
 & \dint_{B^r(T)} u(x,T) dx \le C_2r^{1/\alpha} \quad \text{ if  } \alpha> 1/N,\\
  & \dint_{B^r(T)} u(x,T) dx \le C_2r^{N} \quad \text{ if  } \alpha< 1/N
\end{aligned}
\]
for  $r \in (0,R_2)$.
By setting $R=\min \{ R_1,R_2 \}$,
the proof for $\alpha \neq 1/N$ is completed.

Finally, we consider the case $\alpha=1/N$.
In this case, we need a more precise estimate for $I_2(r)$
than the proof of Theorem~\ref{th:upper}.
Since $c_1s^{1/N}\leq |\eta(s)|\leq c_2s^{1/N}$ and
$[0,s_0]\setminus \T(2r)\subset [cr^N,s_0]$ for some $c>0$, we have
\[
  \begin{aligned}
    I_2(r)
    &=\int_{s\in[0,s_0] \setminus \T(2r)}
    \int_{B_0^r} \frac{1}{(4\pi s)^{N/2}} {\rm exp}\left( -\frac{|\eta(s)|^2}{16s} \right) dyds
\\
    &\leq \int_{s\in[0,s_0] \setminus \T(2r)}
    \int_{B_0^r} \frac{1}{(4\pi s)^{N/2}} {\rm exp}\left( -cs^{2/N-1} \right) dyds
\\
    &\leq Cr^N\int_{cr^N}^{\infty}\frac{1}{(4\pi s)^{N/2}} {\rm exp}\left( -cs^{2/N-1} \right) ds
\\
    &\leq Cr^N.
  \end{aligned}
\]
This completes the proof.
\end{proof}

 \Sect{Fractional Brownian motion} \label{sect:Brown}
 In this section, we consider the case  where $\xi(t)$ is a sample path
of the fractional Brownian motion with the Hurst exponent $H$ with $0<H\le 1/2$.
We sometimes write $\xi_t$ for $\xi(t)$.
We denote by $E$ and $P$ the expectation and the probability, respectively,
of the fractional Brownian motion starting from the origin.
We first show that the Lebesgue measure $\tau(r)$
of  $\T (\rho)$ defined by \eqref{eq:bigtau}
has the following properties:

 \begin{pro}\label{pA1}
 Suppose that $\xi(t)$ is a sample path
 of the fractional Brownian motion  with the Hurst exponent $H$.
 Then
 for any $\delta>0$, there exists a constant  $r_0(\omega)>0$ such that
 \begin{align*}
 \sigma(r) \ge \{\log\log (1/r)\}^{-1/(2H)-\delta} r^{1/H}
\end{align*}
for  $r \in (0,r_0]$ with probability one.
\end{pro}

\begin{proof}
Let
\[
A_m
:=
\big \{ \sigma(e^{-m}) \le (\log m)^{-1/(2H)-\delta/2} e^{-m/H}  \big \}.
\]
By using  (3.14) of \cite{W}, we have
\[
P(A_m)\le C\exp(-c (\log m)^{1+\delta H}).
\]
By the Borel-Cantelli lemma, we have
\[
P(\limsup_{m\to \infty} A_m)=0.
\]
Then  for $e^{-m-1} <r \le e^{-m}$, $\sigma$ satisfies
\[
\sigma(r)
\ge  \sigma(e^{-m-1})
\ge (\log (m+1))^{-1/(2H)-\delta/2} e^{-(m+1)/H}
  \ge \{ \log \log (1/r)\}^{-1/(2H)-\delta} r^{1/H}
\]
 for all sufficiently large $m>0$ with probability one.
Therefore, we have
\[
\sigma(r)   \ge \{ \log\log  (1/r)\}^{-1/(2H)-\delta} r^{-1/H}
\]
for all sufficiently small  $r>0$ with probability one.
This proves the desired result.
\end{proof}

 \begin{pro}\label{pA2}
 Suppose that $\xi(t)$ is a sample path
 of the fractional Brownian motion  with the Hurst exponent $H$.
 Then the following holds:
 \begin{description}
\item[\rm (i)]
If $H >1/N$,
for any $\delta>0$, there exists a constant  $r_0(\omega)>0$ such that
\[
\tau(r) \leq  r^{1/H} \{ \log \log (1/r)\} ^{1+\delta}
\]
 for  $r \in (0,r_0]$ with probability one.

\item[\rm (ii)]
If $H=1/N$,  for any $\delta>0$, there
exists a constant  $r_0(\omega)>0$  such that
\[
\tau(r) \leq  r^N ( \log (1/r)) \{ \log \log (1/r)\}^{1+\delta}
\]
 for  $r \in (0,r_0]$ with probability one.

\item[\rm (iii)]
If $H<1/N$,
for any $\delta>0$, there
exists a constant  $r_0(\omega)>0$ such that
\[
\tau(r) \leq  r^{N} \{ \log \log (1/r)\}^{1+\delta}
\]
 for  $r \in (0,r_0]$ with probability one.
\end{description}
\end{pro}

Before giving a proof of Proposition~\ref{pA2},
we complete the proof of Theorem~\ref{th:Brown}.

\begin{proof}[Proof of Theorem~\ref{th:Brown}.]
 By Theorem \ref{th:lower} and Proposition \ref{pA1},
 we obtain the lower bound immediately.
 By Proposition \ref{pA2} and a similar computation
 to the proof of Corollary~\ref{co:bound}, we have
\begin{align*}
 &r^N \int_r^\infty  l^{-N-1} \tau(l) dl  \le Cr^{1/H}\{ \log \log (1/r)\}^{1+\delta}
  \quad \text{ if  } H> 1/N,\\
& r^N \int_r^\infty  l^{-N-1} \tau(l) dl  \le Cr^N \log(1/r)^2 \{ \log \log (1/r)\}^{1+\delta}
   \quad \text{ if  } H= 1/N,\\
&r^N \int_r^\infty  l^{-N-1} \tau(l) dl  \le Cr^{N}\log(1/r) \{ \log \log (1/r)\}^{1+\delta}
  \quad \text{ if  } H< 1/N.
\end{align*}
Thus the desired  bounds are obtained.
 \end{proof}

To show Proposition~\ref{pA2}, we need some preparations.
We consider $ f(x)$ satisfying
\[
 f(x)\ge 0,  \quad  \| f\|_1<\infty,  \quad  \|f\|_\infty<\infty,
\]
and define
\[
   Z:=\dint_0^\infty f (\xi_s)ds,  \qquad Z_t:=\dint_0^t f(\xi_s )ds.
\]

\begin{pro}\label{p1} \
 \begin{description}
\item[\rm (i)]
If  $H>1/N$, then
\[
E[Z^n]\le n! C_1^n
\]
for all $n\in \N$, where $C_1:=\sqrt{\frac{2}{\pi}}  \frac{1}{H N-1} \| f \|_1+\| f\|_\infty$.
\item[\rm (ii)]
If  $H=1/N$, then
\[
E[Z_t^n]\le n! C_2^n (\log t)^n
\]
for all $n\in \N$,
where $C_2:=\sqrt{\frac{2}{\pi}}  \| f \|_1+\| f\|_\infty$.
\item[\rm (iii)]
If  $0<H<1/N$, then
\[
E[Z_t^n]\le n! C_3^n t^{(1-H  N)n}
\]
for all $n\in \N$,
where $C_3:=\sqrt{\frac{2}{\pi}}  \frac{1}{1-H N} \| f \|_1$.
\end{description}
\end{pro}

We can write
\[
E[Z^n_t]
=\int_{0<t_1<t} \ldots \int_{0<t_n<t}
E[f (\xi_{t_1}) \ldots f(\xi_{t_n})]
dt_1 \ldots dt_n
= n! a_n(t),
\]
where
\[
a_n(t)
:=\int \ldots \int_{0< t_1<\cdots <t_n<t}
E[ f(\xi_{t_1}) \ldots f(\xi_{t_n})]
dt_1 \ldots dt_n.
\]
In addition, let
\begin{align*}
b_n(t)
:=&\int \ldots \int_{\substack{0< t_1<\cdots <t_n<t,\\ t_j-t_{j-1}>1}}
E[f(\xi_{t_1}) \ldots f(\xi_{t_n})]
dt_1 \ldots dt_n,\\
b_0(t):=&1,\\
t_0:=&0,
\end{align*}
and $a_n:=a_n(\infty)$, $b_n:=b_n(\infty)$.

To show Proposition \ref{p1}, we give the following lemmas.
\begin{lem}\label{lem2}
The density of $\xi_{t_1}$, $\xi_{t_2}-\xi_{t_1}$, \ldots, $\xi_{t_n}-\xi_{t_{n-1}}$ satisfies
\[
p_N(x_1,\ldots x_n)
\le
  \bigg(\frac{2}{\pi}\bigg)^{nN/2} \prod_{j=1}^n(t_j-t_{j-1})^{-H N}.
\]
\end{lem}

\begin{proof}
Let $L_n(t_1,\ldots, t_n)$ be the covariance matrix of
$\xi^1_{t_1}$, $\xi^1_{t_2}-\xi^1_{t_1}$, \ldots, $\xi^1_{t_n}-\xi^1_{t_{n-1}}$,
where $\xi^1$ is a sample path of the one-dimensional fractional Brownian motion.

By \cite[Lemma 3.3]{CS},  we have
\[
\text{det}L_n(t_1,\ldots, t_n)
\ge \frac{1}{2^n} \prod_{j=1}^n(t_j-t_{j-1})^{2H},
\]
where $t_0=0$.
Then the density of $\xi^1_{t_1}$, $\xi^1_{t_2}-\xi^1_{t_1}$,
\ldots , $\xi^1_{t_n}-\xi^1_{t_{n-1}}$ satisfies
\[
p(x_1,\ldots x_n)
\le
 \bigg(\frac{2}{\pi}\bigg)^{n/2} \prod_{j=1}^n(t_j-t_{j-1})^{-H},
\]
and hence we obtain the desired result.
\end{proof}

\begin{lem}\label{lem1}
If $H>1/N$, then
\[
b_n \le n! \hat{C}_1^n
\]
for all $n\in \N$, where $\hat{C}_1=\sqrt{\frac{2}{\pi}}  \frac{1}{H N-1} \| f \|_1$.
\end{lem}

\begin{proof}
First we note that
\[
\begin{aligned}
E[f(\xi_{t_1}) \ldots f (\xi_{t_n})]
& \le \bigg(\frac{2}{\pi}\bigg)^{nN/2}  \prod_{j=1}^n(t_j-t_{j-1})^{-H N}
\\
&\qquad  \times
\int_{\mathbb{R}^N} \ldots \int_{\mathbb{R}^N}
f (x_1)f(x_1+x_2)\ldots f(x_1+\cdots +x_n)
dx_1 \ldots dx_n\\
& =  \bigg(\frac{2}{\pi}\bigg)^{nN/2}  \prod_{j=1}^n(t_j-t_{j-1})^{-H N}
\| f \|_1^n.
\end{aligned}
\]
Then we have
\[
\begin{aligned}
b_n
&  \le \bigg(\frac{2}{\pi}\bigg)^{nN/2}
\| f \|_1^n
\int \ldots \int_{\substack{ 0< t_1<\cdots <t_n, \\ t_j-t_{j-1}>1}}
 \prod_{j=1}^n(t_j-t_{j-1})^{-H N}
dt_1 \ldots dt_n\\
&  =  \bigg(\frac{2}{\pi}\bigg)^{nN/2}
\| f \|_1^n
(H N-1)^{-n}.
\end{aligned}
\]
Hence we obtain the desired inequality.
\end{proof}

\begin{lem}\label{lem3}
If $H=1/N$, then
\[
b_n(t) \le n! \hat{C}_2^n(\log t)^n
\]
for all $n\in \N$, where $\hat{C}_2=\sqrt{\frac{2}{\pi}}  \| f \|_1 $.
\end{lem}

\begin{proof}
It suffices to show
$
a_n\le C_1^n.
$
Since
\[
E[f(\xi_{t_1}) \ldots f(\xi_{t_n})]\\
\le  \bigg(\frac{2}{\pi}\bigg)^{nN/2}  \prod_{j=1}^n(t_j-t_{j-1})^{-H N}
\| f \|_1^n,
\]
we have
\[
\begin{aligned}
b_n(t)
\le & \bigg(\frac{2}{\pi}\bigg)^{nN/2}
\| f \|_1^n
\int \ldots \int_{\substack{ 0< t_1<\cdots <t_n<t, \\t_j-t_{j-1}>1}}
 \prod_{j=1}^n(t_j-t_{j-1})^{-H N}
dt_1 \ldots dt_n\\
\le &  \bigg(\frac{2}{\pi}\bigg)^{nN/2}
\| f \|_1^n
(\log t)^n.
\end{aligned}
\]
Hence we obtain the desired result.
\end{proof}

Now, let us prove Proposition \ref{p1} by using the above lemmas.

\begin{proof}
[Proof of Proposition \ref{p1}]
First we show (i).
When $n=1$,
\[
\begin{aligned}
a_1
&=  \int_0^\infty E[f(\xi_{t_1})]dt_1\\
&= \int_0^1 E[f(\xi_{t_1})]dt_1
+\int_1^\infty E[f(\xi_{t_1})]dt_1\\
& \le \|f\|_\infty +b_1
\le \|f\|_\infty +\hat{C}_1
\le C_1.
\end{aligned}
\]
When $n=2$,
\[
\begin{aligned}
a_2
&=\int \int_{ 0<t_1<t_2} E[f(\xi_{t_1}) f(\xi_{t_2})]dt_1dt_2\\
&=\int \int_{\substack{ 0<t_1<t_2, t_1<1,\\ t_2-t_1<1}} E[f(\xi_{t_1}) f(\xi_{t_2})]dt_1dt_2
+\int \int_{\substack{ 0<t_1<t_2-1,\\ t_1<1}} E[f(\xi_{t_1}) f(\xi_{t_2})]dt_1dt_2\\
& \qquad +\int \int_{\substack{ 1\le t_1<t_2,
\\t_2-t_1<1}} E[f(\xi_{t_1}) f(\xi_{t_2})]dt_1dt_2
+\int \int_{1\le t_1<t_2-1} E[f(\xi_{t_1}) f(\xi_{t_2})]dt_1dt_2\\
& \le \|f\|_\infty^2 +2\|f\|_\infty b_1+b_2\\
& \le \|f\|_\infty^2 +2\|f\|_\infty \hat{C}_1+\hat{C}_1^2\\
& \le C_1^2.
\end{aligned}
\]
For general $n$, by the same computation as above, we have
\begin{align*}
a_n
\le  \sum_{k=0}^n {}_nC_k \| f \|_\infty^k \hat{C}_1^{n-k}
=(\|f \|_\infty+\hat{C}_1 )^n
=C_1^n
\end{align*}
and hence we obtain (i).

By the same way as the proof of  (i),
we have (ii) with the aid of Lemma \ref{lem3}.
Moreover, from
\[
E[f(\xi_{t_1}) \ldots f(\xi_{t_n})]
\le   \bigg(\frac{2}{\pi}\bigg)^{nN/2}  \prod_{j=1}^n(t_j-t_{j-1})^{-H N}
\| f \|_1^n,
\]
we see that
\[
\begin{aligned}
a_n(t)
& \le  \bigg(\frac{2}{\pi}\bigg)^{nN/2}
\| f \|_1^n
\int \ldots \int_{0< t_1<\cdots <t_n<t}
 \prod_{j=1}^n(t_j-t_{j-1})^{-H N}
dt_1 \ldots dt_n\\
& \le   \bigg(\frac{2}{\pi}\bigg)^{nN/2}
\| f \|_1^n
(1-H N)^{-n}
 t^{(1-H N)n}.
\end{aligned}
\]
Therefore, (iii) holds.
\end{proof}

\begin{lem}\label{le:1} \
 \begin{description}
\item[\rm (i)]
If $H>1/N$, then there exists $C<\infty$ such that
\[
E[\tau(e^{-m})^n]\le n! C^ne^{-m n/H}
\]
for all $m,n\in \N$.
\item[\rm (ii)] If $H=1/N$, then  there exists $C<\infty$ such that
\[
E[\tau(e^{-m})^n]\le n! C^n e^{-m nN} m^n
\]
for all $m,n\in \N$.
\item[\rm (iii)] If $H<1/N$, then there exists $C<\infty$ such that
\[
E[\tau(e^{-m})^n]\le n! C^n e^{-mNn}
\]
for all $m,n\in \N$.
\end{description}
\end{lem}

\Proof
Note that by the stationary increment property of the fractional Brownian motion,
\[
\tau(e^{-m})
=_d \int_0^{s_0} 1_{\{|\xi(s)| \le e^{-m}\} }ds,
\]
where $X=_d Y$ means that $X$ has same distribution as $Y$.
In addition, by the scaling property,  we have
\begin{align*}
&\int_0^{s_0} 1_{\{|\xi(s)| \le e^{-m}\} }ds
=_d e^{-m/H} \int_0^{s_0 \exp(m/H)} 1_{\{|\xi(s)| \le 1\} }ds.
\end{align*}
Note that $\|  f \| _1 <\infty$  and $\| f\|_\infty <\infty$ hold.
Thus, if we substitute $1_{\{|x|\le 1\}}$ for $f$ in Proposition \ref{p1},
we obtain the desired result.
\Qed

Finally, we give the following lemma.

\begin{lem}\label{ine1} \
\begin{description}
\item[\rm (i)]
If $H>1/N$,  there exist $C, c>0$ such that
\[
P( \tau(e^{-m}) \ge M e^{-m/H}  )
 \le  Ce^{-cM}
\]
for any $M<\infty$ and  $m\in \N$.
\item[\rm (ii)]
If $H=1/N$,  there exist $C, c>0$ such that
\[
P( \tau(e^{-m}) \ge M e^{-m N} m )
 \le  Ce^{-cM}
\]
for any $M<\infty$ and  $m\in \N$.
\item[\rm (iii)]
  If $H<1/N$,  there exist $C, c>0$ such that
\[
P( \tau(e^{-m}) \ge M e^{-mN} )
 \le  Ce^{-cM}
\]
for any $M<\infty$ and  $m\in \N$.
\end{description}
\end{lem}

\begin{proof}
First, we show (i).
By Lemma \ref{le:1}, there exist
 $c>0$ and $C<\infty$ such that for any $m\in \N$
\[
E[\exp(c\tau(e^{-m}) e^{m/H} ) ]< C.
\]
Thus, by Chebyshev's inequality,
we obtain
\begin{align*}
P(\tau(e^{-m}) \ge M e^{-m/H})
=&P(\exp(c\tau(e^{-m}) e^{m/H}) \ge e^{cM} )
\\
\le&Ce^{-cM} E[\exp(c\tau(e^{-m}) e^{m/H}) ]
\\
\le&Ce^{-cM}.
\end{align*}
Similarly, we have (ii) and (iii).
\end{proof}


Now, we are in a position to show the main proposition in this section.

\begin{proof}[Proof of Proposition \ref{pA2}]
First, we consider the case of $H>1/N$.
Let
\[
A_m
:=
\{ \tau(e^{-m})  \ge (\log m)^{1+\delta/2} e^{-m/H}  \}.
\]
By Lemma \ref{ine1},
\[
P(A_m)\le C\exp(-c (\log m)^{1+\delta/2}).
\]
By the Borel-Cantelli lemma, we have
\[
P(\limsup_{m\to \infty} A_m)=0.
\]
Hence  if $e^{-m-1} <r \le e^{-m}$, then
\[
\tau(r)
\le \tau (e^{-m})
\le(\log m)^{1+\delta/2} e^{-m/H}
\le \{ \log \log (1/r)\}^{1+\delta} r^{1/H}
\]
for all sufficiently large $m$ with probability one.
Thus we obtain
\[
\tau(r) \le \{ \log \log (1/r)\}^{1+\delta} r^{1/H}
\]
for all sufficiently small  $r>0$ with probability one.
This proves (i).

The assertions  (ii) and (iii) can be proved in the  same way.
\end{proof}


\

\noindent
{\bf Acknowledgements.}
MF was partly supported by Grant-in-Aid for JSPS Research Fellow (No.~JP20J20941).
IO was supported by
JSPS  KAKENHI Grant-in-Aid  for Early-Career Scientists (No.~JP20K14329).
EY was partly supported by
JSPS  KAKENHI Grant-in-Aid for Scientific Research (A)  (No.~JP17H01095).


\end{document}